\documentclass[11pt]{amsart}
\usepackage[T1]{fontenc}
\usepackage[utf8]{inputenc}
\usepackage{lmodern}
\usepackage{microtype}
\usepackage{xcolor}
\usepackage{amsmath,amssymb,amsthm,mathtools}
\usepackage{enumitem}
\usepackage[margin=1in]{geometry}
\usepackage{hyperref}
\hypersetup{colorlinks=true,linkcolor=blue,citecolor=blue,urlcolor=blue}

\theoremstyle{plain}
\newtheorem{theorem}{Theorem}[section]
\newtheorem{lemma}[theorem]{Lemma}
\newtheorem{proposition}[theorem]{Proposition}
\newtheorem{corollary}[theorem]{Corollary}
\theoremstyle{definition}

\theoremstyle{remark}
\newtheorem{remark}[theorem]{Remark}

\DeclareMathOperator{\End}{End}
\DeclareMathOperator{\chr}{char}
\newcommand{\FF}{\mathbb{F}}
\newcommand{\KK}{\mathbb{K}}
\newcommand{\PP}{\mathbb{P}}
\newcommand{\ZZ}{\mathbb{Z}}
\newcommand{\JJ}{\mathcal{J}}
\newcommand{\cA}{\mathcal{A}}
\DeclareMathOperator{\GL}{GL}
\numberwithin{equation}{section}

\title[Jordan rigidity of full matrix algebras]{Jordan rigidity of full matrix algebras}

\author{Ilja Gogi\'{c}}
\address{I.~Gogi\'{c}, Department of Mathematics, Faculty of Science, University of Zagreb, Bijeni\v{c}ka 30, 10000 Zagreb, Croatia}
\email{ilja@math.hr}

\author{Matija Kazalicki}
\address{M.~Kazalicki, Department of Mathematics, Faculty of Science, University of Zagreb, Bijeni\v{c}ka 30, 10000 Zagreb, Croatia}
\email{matija.kazalicki@math.hr}

\author{Mateo Toma\v{s}evi\'{c}}
\address{M.~Toma\v{s}evi\'{c}, Department of Mathematics, Faculty of Science, University of Zagreb, Bijeni\v{c}ka 30, 10000 Zagreb, Croatia}
\email{mateo.tomasevic@math.hr}



\date{\today}

\begin{document}
	
\begin{abstract}
	Let $\mathbb{F}$ be a field of characteristic different from $2$, and let
	$M_n(\mathbb{F})^+$ denote the Jordan algebra of all $n\times n$ matrices
	over $\mathbb{F}$ with product $X\circ Y:=(XY+YX)/2$. We prove a rigidity
	theorem for $M_n(\mathbb{F})^+$, $n\ge2$: if $\mathcal{J}$ is any $2$-torsion-free Jordan ring and $\phi:M_n(\mathbb{F})^+\to\mathcal{J}$ is a Jordan multiplicative (product-preserving) map, then $\phi(0)$ is an idempotent and $X\mapsto\phi(X)-\phi(0)$ is either zero or an injective Jordan ring homomorphism. Thus, up to an idempotent constant, preservation of the Jordan product alone forces additivity and the zero-or-injective dichotomy. When specialized to associative codomains, the theorem yields the Jacobson--Rickart decomposition into homomorphic and antihomomorphic parts. In particular, for maps $M_n(\mathbb F)^+\to M_k(\mathbb K)^+$, where $\mathbb K$ is a field of characteristic different from $2$, we also obtain a block normal form governed by finite-dimensional $\mathbb K$-representations of $\mathbb F$, together with a criterion for the existence of nonconstant maps.
\end{abstract}
	
	\maketitle
	
	\section{Introduction}
	
	Jordan rings form an important class of nonassociative rings, closely related
	to associative rings but governed by their own structural theory. By a
	\emph{Jordan ring} we mean an additive abelian group $\JJ$ equipped with a
	biadditive commutative product $(x,y)\mapsto x\circ y$ satisfying the
	\emph{Jordan identity}
	\[
	x^2\circ(x\circ y)=x\circ(x^2\circ y),
	\qquad x,y\in\JJ,
	\]
	where $x^2:=x\circ x$. We say that a Jordan ring is
	\emph{$2$-torsion-free} if $2x=0$ implies $x=0$.
	
	Let $\ZZ[1/2]$ denote the ring of dyadic rationals. If $\cA$ is an associative
	$\ZZ[1/2]$-algebra, then $\cA$ gives rise to the corresponding
	\emph{special Jordan ring} $\cA^+$, with the same additive group and
	product
	\begin{equation}\label{eq:J-product}
		x\circ y:=\frac12(xy+yx),
		\qquad x,y\in\cA.
	\end{equation}
	A map $\phi:\JJ_1 \to \JJ_2$ between Jordan rings $\JJ_1$ and $\JJ_2$ is called
	\emph{Jordan multiplicative} if
	\begin{equation}\label{eq:JM}
		\phi(x\circ y)=\phi(x)\circ\phi(y),
		\qquad x,y\in \JJ_1.
	\end{equation}
	No additivity, scalar-linearity, or further compatibility condition is
	assumed.
	
	The standard morphisms between Jordan rings are the \emph{Jordan ring homomorphisms}, that is, additive maps satisfying \eqref{eq:JM}. In the special Jordan setting \eqref{eq:J-product}, where an underlying associative ring is present, ordinary ring homomorphisms and antihomomorphisms give the basic examples. A classical problem is to decide when every Jordan ring homomorphism is of one of these two types, or more generally splits into homomorphic and antihomomorphic parts. Foundational results in this direction are due to Jacobson--Rickart, Herstein, and Smiley \cite{JacobsonRickart,Herstein,Smiley}; for a recent survey, see \cite{BresarZelmanov}.
	
	A parallel line of research concerns automatic additivity. Here one starts with a
	product-preserving map which is not assumed to be additive, and tries to recover additivity from additional hypotheses, usually bijectivity or some form of nondegeneracy. A fundamental result in this direction is Martindale's theorem: every multiplicative bijection from a prime ring containing a nontrivial idempotent onto an arbitrary ring is additive \cite{Martindale}. Jordan-type versions of this problem have been studied, for example, in \cite{DolinarKuzmaStopar,GTmatrix,GTstructural,GTsemi,Ji,Lu,Molnar}.
	
	The present paper combines these two viewpoints for full matrix Jordan algebras. Let $\FF$ be a field of characteristic different from $2$, and let $M_n(\FF)$ be the algebra of all $n\times n$ matrices over $\FF$. We study Jordan multiplicative maps from $M_n(\FF)^+$ into arbitrary $2$-torsion-free Jordan rings. Unlike most existing automatic-additivity results, we impose neither bijectivity nor any nondegeneracy assumption on the map, and we impose no codomain hypothesis beyond $2$-torsion-freeness. This extends the self-map case treated in \cite{GTmatrix}, where automatic additivity and the
	explicit form of zero-preserving Jordan multiplicative self-maps of $M_n(\FF)^+$ were obtained. Our main result says that, for $n\ge2$, every such map becomes a Jordan ring
	homomorphism after subtracting its value at zero. Since $0\circ0=0$, this value is necessarily an idempotent. For an idempotent $e$ in a Jordan ring
	$\JJ$, set
	\[
	\JJ_0(e):=\{x\in\JJ:e\circ x=0\},
	\]
	the Peirce $0$-space of $e$. If $\JJ$ is $2$-torsion-free, then $\JJ_0(e)$ is a Jordan subring.
	
	\begin{theorem}[Main theorem]\label{thm:general}
		Let $\FF$ be a field with $\chr(\FF)\ne2$, let $n\ge2$, and let
		$\JJ$ be a $2$-torsion-free Jordan ring. A map
		$\phi:M_n(\FF)^+\to\JJ$ is Jordan multiplicative if and only if
		$e:=\phi(0)$ is an idempotent and there exists a Jordan ring homomorphism $\psi:M_n(\FF)^+\to\JJ_0(e)$ such that
		\[
		\phi(X)=e+\psi(X),
		\qquad X\in M_n(\FF).
		\]
		In this case, $\psi$ is either zero or injective. Consequently, $\phi$ is
		either constant idempotent-valued or injective.
	\end{theorem}
	
	We call this phenomenon \emph{Jordan rigidity}: preserving the Jordan product
	already forces a Jordan ring homomorphism after the idempotent value at zero
	is removed, and simplicity of $M_n(\FF)^+$ makes that homomorphism either
	zero or injective.
	
	The assumption $n\ge2$ cannot be omitted. If $\chr(\FF)\ne2$, the map
	$\FF\to\FF$ sending $0$ to $0$ and every nonzero element to $1$ is
	zero-preserving and multiplicative, hence Jordan multiplicative for the
	one-dimensional Jordan algebra $\FF^+$, but it is not additive.
	
	The proof of Theorem~\ref{thm:general} is given in
	Section~\ref{sec:proof-main}. We first consider the case in which the
	codomain is a Jordan algebra over $\ZZ[1/2]$. After subtracting the value
	at zero, the problem reduces to a zero-preserving Jordan multiplicative
	map $\psi:M_n(\FF)^+\to\JJ$. A Peirce-space calculation in the upper-left
	copy of $M_2(\FF)^+$ gives the initial additive relation
	\[
	\psi(E_{11}+E_{22})=\psi(E_{11})+\psi(E_{22}),
	\]
	where $E_{11}$ and $E_{22}$ are the first two diagonal matrix units. The
	main external input is \cite[Theorem~1.1]{GKT}, which identifies the
	semigroup generated by the Jordan multiplication operators on
	$M_n(\FF)^+$ with $\End_\FF(M_n(\FF))$. This theorem transports the single
	relation above to arbitrary pairs of matrices, yielding additivity of
	$\psi$. The general $2$-torsion-free case follows by localizing the
	additive group at $2$; the same argument also shows that $\JJ_0(e)$ is a
	Jordan subring. Finally, once additivity is known, the zero-or-injective dichotomy follows from the Jordan-ring simplicity of $M_n(\FF)^+$.
	
	Section~\ref{sec:assoc-matrix} specializes the main result to associative
	codomains, and then to full matrix algebras. In this setting, the
	Jacobson--Rickart theorem yields a decomposition of the nonconstant part
	into homomorphic and antihomomorphic components.
	
	\begin{theorem}\label{thm:associative}
		Let $\FF$ be a field with $\chr(\FF)\ne2$, let $n\ge2$, let
		$\cA$ be a unital associative algebra over $\ZZ[1/2]$, and let
		$\phi:M_n(\FF)^+\to\cA^+$ be Jordan multiplicative. Set
		$e:=\phi(0)$ and $\psi:=\phi-e$. Then $e$ is an idempotent and
		$\psi$ takes its values in the associative corner $(1-e)\cA(1-e)$. Moreover, set $p:=\psi(I_n)$. Then there exists an idempotent $f$ in the
		associative subring of $p\cA p$ generated by $\psi(M_n(\FF))$, central in that
		subring, such that the map
		\[
		\theta:M_n(\FF)\to f\cA f, \qquad \theta(X) := f \psi(X)
		\]
		is a unital ring homomorphism, the map
		\[
		\eta:M_n(\FF)\to(p-f)\cA(p-f), \qquad \eta(X):=(p-f)\psi(X)
		\]
		is a unital ring antihomomorphism, and $\psi=\theta+\eta$. Consequently,
		$\phi=e+\theta+\eta$.
	\end{theorem}
	
	For matrix codomains, this decomposition can be made completely explicit. The homomorphic part is put into standard matrix-unit form using Morita equivalence for full matrix rings \cite[Section~17B]{LamLectures}, while the antihomomorphic part is treated in the same way after transposition.
	
	\begin{theorem}\label{thm:matrix}
		Let $\FF$ and $\KK$ be fields with characteristics different from $2$.
		Let $n\ge2$, $k\ge1$, and let $\phi:M_n(\FF)^+\to M_k(\KK)^+$ be
		Jordan multiplicative. Then there exist nonnegative integers $s,r,u,t$ with
		$s+nr+nu+t=k$, and an invertible matrix $S\in\GL_k(\KK)$, such that
		\begin{equation}\label{eq:normal}
			S\phi(X)S^{-1}
			=
			I_s\oplus\omega_n(X)\oplus\sigma_n(X^{\mathsf{T}})\oplus0_t,
			\qquad X\in M_n(\FF),
		\end{equation}
		where $X^{\mathsf T}$ denotes the transpose of $X$, and direct summands of
		size $0$ are omitted.
		\begin{itemize}
			\item If $r>0$, then $\omega:\FF\to M_r(\KK)$ is a unital ring
			homomorphism, and $\omega_n:M_n(\FF)\to M_n(M_r(\KK))\cong M_{nr}(\KK)$ is
			its $n$-fold amplification, given by
			$\omega_n([x_{ij}]):=[\omega(x_{ij})]$.
			\item If $u>0$, then $\sigma:\FF\to M_u(\KK)$ is a unital ring
			homomorphism, and $\sigma_n$ is defined analogously.
		\end{itemize}
		Conversely, every map of this form is Jordan multiplicative. Moreover,
		$\phi(0)=0$ if and only if $s=0$, and $\phi$ is unital if and only if
		$t=0$.
	\end{theorem}
	
	When specialized to self-maps over the same field, Theorem~\ref{thm:matrix}
	recovers the classification from \cite{GTmatrix}. In the general case, the
	new coefficients that appear in \eqref{eq:normal} are precisely the
	finite-dimensional $\KK$-representations of $\FF$. If
	$\chr(\FF)\ne\chr(\KK)$, no nonzero coefficient block can occur; hence every
	Jordan multiplicative map $M_n(\FF)^+\to M_k(\KK)^+$ is constant
	idempotent-valued. 
	
	Finally, if $\FF$ and $\KK$ have the same characteristic, let
	$\PP$ denote their common prime subfield. Then finite-dimensional unital
	representations of the $\PP$-algebra $\FF$ on $\KK$-vector spaces are
	equivalent to finite-dimensional unital modules over the commutative
	$\KK$-algebra $\KK\otimes_{\PP}\FF$, where $\KK$ acts through the first
	tensor factor. If $d_\KK(\FF)$ is the least $\KK$-dimension of a nonzero such
	module, with $d_\KK(\FF)=\infty$ when none exists, then nonconstant maps
	$M_n(\FF)^+\to M_k(\KK)^+$ exist exactly when $k\ge n\,d_\KK(\FF)$. We recall
	this description and compute $d_\KK(\FF)$ in the finite separable and
	finite-field cases.
	
	\section{Preliminaries}\label{sec:preliminaries}
	
	Throughout the paper, $\FF$ denotes a field with $\chr(\FF)\ne2$. For
	$n\ge1$, we denote by $I_n$ the identity matrix in $M_n(\FF)$ and by
	$E_{ij}$ the standard matrix units.
	
	A \emph{Jordan algebra over $\ZZ[1/2]$} is a Jordan ring whose additive group
	is a $\ZZ[1/2]$-module and whose product is $\ZZ[1/2]$-bilinear. If $\JJ$ is
	such an algebra and $a\in\JJ$, we write
	\[
	L_a^\JJ(x):=a\circ x,
	\qquad x\in\JJ,
	\]
	for the \emph{Jordan multiplication operator} by $a$. When $\JJ$ is clear from the
	context, we simply write $L_a$.
	
	The main external input is the following theorem for $M_n(\FF)^+$, proved in
	\cite[Theorem~1.1]{GKT}.
	
	\begin{theorem}\label{thm:GKT}
		Let $n\ge1$. Every $\FF$-linear endomorphism of $M_n(\FF)$ is a finite
		composition of Jordan multiplication operators. Equivalently, the semigroup
		generated by the operators $L_A$, $A\in M_n(\FF)$, is the full endomorphism
		semigroup $\End_\FF(M_n(\FF))$.
	\end{theorem}
	
	\begin{remark}\label{rem:jordan-simple}
		We shall use that $M_n(\FF)^+$ is simple as a Jordan ring. This is a special
		case of Herstein's Jordan-ring simplicity theorem: if $\mathcal R$ is a
		simple associative ring of characteristic different from $2$, then
		$\mathcal R^+$ is simple as a Jordan ring
		\cite[Theorem~1]{HersteinSimple}. In the present matrix setting, the same
		conclusion also follows directly from Theorem~\ref{thm:GKT}. Indeed, let $\mathcal I$ be a nonzero Jordan ideal of $M_n(\FF)^+$, that is,
		an additive subgroup invariant under all maps $L_A:X\mapsto A\circ X$,
		$A\in M_n(\FF)$. Choose $0\ne X\in\mathcal I$. By
		Theorem~\ref{thm:GKT}, the semigroup generated by the operators $L_A$ is
		$\End_\FF(M_n(\FF))$; hence $\mathcal I$ is invariant under every
		$\FF$-linear endomorphism of $M_n(\FF)$. For any $Y\in M_n(\FF)$, choose
		$\Theta\in\End_\FF(M_n(\FF))$ with $\Theta(X)=Y$. Then
		$Y=\Theta(X)\in\mathcal I$, so $\mathcal I=M_n(\FF)$.
	\end{remark}
	
	We recall the Peirce notation and multiplication rules used below. Let $\JJ$
	be a Jordan algebra over $\ZZ[1/2]$, and let $e\in\JJ$ be an idempotent.
	The Peirce spaces of $\JJ$ relative to $e$ are the eigenspaces of the
	corresponding Jordan multiplication operator $L_e$:
	\[
	\JJ_\lambda(e):=\{x\in\JJ:e\circ x=\lambda x\},
	\qquad \lambda\in\{1,\tfrac12,0\}.
	\]
	They give the \emph{Peirce decomposition}
	\[
	\JJ=\JJ_1(e)\oplus\JJ_{1/2}(e)\oplus\JJ_0(e).
	\]
	This is an algebraic direct-sum decomposition; no finite-dimensionality is
	assumed, and some summands may be zero. The spaces $\JJ_1(e)$ and $\JJ_0(e)$
	are Jordan subalgebras, and $\JJ_1(e)\circ\JJ_0(e)=0$.
	
	Idempotents $e_1,\ldots,e_m$ are pairwise orthogonal if
	$e_i\circ e_j=0$ whenever $i\ne j$. If $e_1,\ldots,e_m$ are pairwise
	orthogonal idempotents with sum $e$, then $e$ is an idempotent and
	$\JJ_1(e)$ is a unital Jordan algebra with unit $e$. The simultaneous Peirce
	decomposition relative to $e_1,\ldots,e_m$ is taken inside $\JJ_1(e)$. Inside
	$\JJ_1(e)$, set
	\[
	\JJ_{ii}
	:=
	\{x\in\JJ_1(e): e_i\circ x=x,\ e_j\circ x=0\text{ for }j\ne i\},
	\]
	and, for $i\ne j$,
	\[
	\JJ_{ij}
	:=
	\{x\in\JJ_1(e): e_i\circ x=e_j\circ x=\tfrac12 x,\ 
	e_\ell\circ x=0\text{ for }\ell\notin\{i,j\}\}.
	\]
	Then the simultaneous Peirce decomposition is, as a direct sum of
	$\ZZ[1/2]$-submodules,
	\begin{equation}\label{eq:Peirce-decomposition}
		\JJ_1(e)=
		\bigoplus_i \JJ_{ii}\oplus
		\bigoplus_{i<j}\JJ_{ij}.
	\end{equation}
	We shall use the standard Peirce multiplication rules:
	\begin{equation}\label{eq:peirce-rules}
		\begin{gathered}
			\JJ_{ii}\circ\JJ_{jj}=0,\qquad
			\JJ_{ii}\circ\JJ_{ij}\subseteq\JJ_{ij},\qquad
			\JJ_{ij}\circ\JJ_{ij}\subseteq\JJ_{ii}\oplus\JJ_{jj}
			\qquad (i\ne j),\\
			\JJ_{ii}\circ\JJ_{jk}=0
			\qquad (i\notin\{j,k\}),\\
			\JJ_{ij}\circ\JJ_{jk}\subseteq\JJ_{ik}
			\qquad (i,j,k\text{ distinct}),\\
			\JJ_{ij}\circ\JJ_{k\ell}=0
			\qquad (\{i,j\}\cap\{k,\ell\}=\emptyset).
		\end{gathered}
	\end{equation}
	See, for example, \cite[Part~II, Chapters~8 and~13]{McCrimmon}. 
	
	We also recall the following standard special-case form of the Peirce
	decomposition.
	\begin{remark}\label{rem:assoc-peirce-zero}
		Let $\cA$ be a unital associative algebra over $\ZZ[1/2]$, and let
		$e\in\cA$ be an idempotent. Then the Peirce spaces of $\cA^+$ relative to
		$e$ are
		\[
		(\cA^+)_1(e)=e\cA e,
		\qquad
		(\cA^+)_{\frac12}(e)=e\cA(1-e)\oplus(1-e)\cA e,
		\qquad
		(\cA^+)_0(e)=(1-e)\cA(1-e).
		\]
		Indeed, write every $a\in\cA$ as
		\[
		a=eae+ea(1-e)+(1-e)ae+(1-e)a(1-e).
		\]
		With respect to the Jordan product in $\cA^+$, the operator
		$L_e:a\mapsto e\circ a$ acts as the identity on $e\cA e$, as multiplication
		by $1/2$ on $e\cA(1-e)\oplus(1-e)\cA e$, and as zero on
		$(1-e)\cA(1-e)$. Hence the displayed corner decomposition identifies these
		three summands with the Peirce $1$, $1/2$, and $0$ spaces, respectively.
	\end{remark}
	
	Finally, we record the localization device used to pass from Jordan algebras
	over $\ZZ[1/2]$ to arbitrary $2$-torsion-free Jordan rings.  Let $\JJ$ be a
	$2$-torsion-free Jordan ring and let
	$S=\{1,2,2^2,\ldots\}\subseteq\ZZ$.  We localize the underlying
	$\ZZ$-module of $\JJ$ at $S$ and write $\JJ[1/2]:=S^{-1}\JJ$. By the standard construction of modules of fractions and its tensor-product description \cite[the construction preceding Proposition~3.3 and
	Proposition~3.5]{AM}, there are canonical isomorphisms
	\[
	S^{-1}\JJ
	\cong
	\ZZ[1/2]\otimes_{\ZZ}\JJ
	\cong
	\JJ\otimes_{\ZZ}\ZZ[1/2].
	\]
	This construction uses only the underlying $\ZZ$-module of $\JJ$ and hence does not require $\JJ$ to be unital.  Every element of $\JJ[1/2]$ can be written as $x/2^m$, with $x\in\JJ$ and $m\ge0$.  Moreover, $x/1=0$ if and only if $2^m x=0$ for some $m\ge0$; since $\JJ$ is $2$-torsion-free, the natural map $\JJ\to\JJ[1/2]$ is injective.
	
	The product on $\JJ$ extends uniquely by $\ZZ[1/2]$-bilinearity to
	$\JJ[1/2]$; explicitly,
	\[
	\frac{x}{2^r}\circ\frac{y}{2^s}
	:=
	\frac{x\circ y}{2^{r+s}}.
	\]
	With this product, $\JJ[1/2]$ is again a Jordan algebra. Indeed, if
	$a=x/2^r$ and $b=y/2^s$, then
	\[
	a^2\circ(a\circ b)-a\circ(a^2\circ b)
	=
	\frac{x^2\circ(x\circ y)-x\circ(x^2\circ y)}{2^{3r+s}}
	=0,
	\]
	because the numerator is zero in $\JJ$. Thus $\JJ[1/2]$ is a Jordan algebra
	over $\ZZ[1/2]$.
	
	\begin{lemma}\label{lem:localization-peirce}
		Let $\JJ$ be a $2$-torsion-free Jordan ring. The canonical localization map
		\[
		\iota:\JJ\to\JJ[1/2],\qquad x\mapsto x/1,
		\]
		is an injective Jordan ring homomorphism. Moreover, if $e\in\JJ$ is an
		idempotent, then $\JJ_0(e)$ is a Jordan subring of $\JJ$.
	\end{lemma}
	
	\begin{proof}
		The map $\iota$ is additive and preserves the Jordan product by construction.
		If $\iota(x)=0$, then $2^m x=0$ for some $m\ge0$. Since $\JJ$ is
		$2$-torsion-free, repeated cancellation of $2$ gives $x=0$. Hence $\iota$ is
		injective.
		
		Let $x,y\in\JJ_0(e)$. We view $x,y$ as elements of $\JJ[1/2]$ through the
		embedding $\iota$. Then $x,y\in(\JJ[1/2])_0(e)$. Since $(\JJ[1/2])_0(e)$ is a
		Jordan subalgebra of $\JJ[1/2]$ by the Peirce multiplication rules, it follows that $x\circ y\in(\JJ[1/2])_0(e)$. Equivalently, $e\circ(x\circ y)=0$ in $\JJ[1/2]$. Since $x\circ y\in\JJ$ and
		$\iota$ is injective, this equality already holds in $\JJ$. Hence
		$x\circ y\in\JJ_0(e)$.
		
		Finally, $\JJ_0(e)$ is an additive subgroup of $\JJ$: it contains $0$, and if
		$x,y\in\JJ_0(e)$, then $e\circ(x-y)=e\circ x-e\circ y=0$. Therefore
		$\JJ_0(e)$ is a Jordan subring of $\JJ$.
	\end{proof}
	
	\section{Proof of the main theorem}\label{sec:proof-main}
	
	We first prove the zero-preserving case. For Jordan algebras over
	$\ZZ[1/2]$, the key point is to establish one additivity relation in the
	upper-left $2\times2$ corner, by means of a Peirce-decomposition argument.
	Once such a relation is known for two linearly independent matrices,
	Theorem~\ref{thm:GKT} transports it to arbitrary pairs of matrices, and hence
	proves additivity of $\phi$. We then pass to arbitrary $2$-torsion-free
	Jordan rings by localization at $2$. Finally, the general case of
	Theorem~\ref{thm:general} is obtained by subtracting the idempotent value
	$\phi(0)$ and applying the zero-preserving result.
	
	\begin{proposition}\label{prop:corner-relation}
		Let $\psi:M_2(\FF)^+\to\JJ$ be Jordan multiplicative, where $\JJ$ is a
		Jordan algebra over $\ZZ[1/2]$, and suppose that $\psi(0)=0$. Then
		\[
		\psi(I_2) = \psi(E_{11}+E_{22})=\psi(E_{11})+\psi(E_{22}).
		\]
	\end{proposition}
	
	\begin{proof}
		Set
		\[
		e_1:=\psi(E_{11}),\qquad
		e_2:=\psi(E_{22}),\qquad
		e:=\psi(I_2).
		\]
		Since $E_{11}$, $E_{22}$ and $I_2$ are idempotents, so are
		$e_1$, $e_2$ and $e$. Moreover,
		\[
		e_1\circ e_2=\psi(E_{11}\circ E_{22})=\psi(0)=0,
		\qquad
		e\circ e_i=\psi(I_2\circ E_{ii})=e_i, \quad i=1,2.
		\]
		Also, since $I_2\circ X=X$ for every $X\in M_2(\FF)$, we have
		\[
		e\circ\psi(X)=\psi(X),\qquad X\in M_2(\FF).
		\]
		Thus the image of $\psi$ is contained in $\JJ_1(e)$.
		Set
		\[
		e_3:=e-e_1-e_2.
		\]
		We claim that $e_3 = 0$. Indeed, first note that $e_3\circ e_1=0$ and $e_3\circ e_2=0$. Moreover,
		\begin{align*}
			e_3\circ e_3
			&=(e-e_1-e_2)\circ(e-e_1-e_2)  \\
			&=e\circ e-2e\circ e_1-2e\circ e_2
			+e_1\circ e_1+2e_1\circ e_2+e_2\circ e_2  \\
			&=e-2e_1-2e_2+e_1+e_2
			=e-e_1-e_2=e_3.
		\end{align*}
		Thus $e_1,e_2,e_3$ are pairwise orthogonal idempotents with sum $e$.
		
		Let 
		\[
		c:=\psi(E_{12}+E_{21}), \qquad  d:=\psi(2(E_{12}+E_{21})).
		\]
		From
		\[
		E_{11}\circ(2(E_{12}+E_{21}))=E_{12}+E_{21}=E_{22}\circ(2(E_{12}+E_{21}))
		\]
		we get
		\begin{equation}\label{eq:e1de2d}
			e_1\circ d=c=e_2\circ d.
		\end{equation}
		Decompose $d \in \JJ_1(e)$ relative to $e_1,e_2,e_3$ as in \eqref{eq:Peirce-decomposition}:
		\[
		d=d_{11}+d_{22}+d_{33}+d_{12}+d_{13}+d_{23}.
		\]
		Then
		\[
		e_1\circ d=d_{11}+\frac12d_{12}+\frac12d_{13},
		\qquad
		e_2\circ d=d_{22}+\frac12d_{12}+\frac12d_{23}.
		\]
		Comparing the direct Peirce components in \eqref{eq:e1de2d} gives
		\[
		d_{11}=d_{22}=d_{13}=d_{23}=0,
		\]
		and hence $c=\frac12d_{12}\in\JJ_{12}$. Since 
		\[
		(E_{12}+E_{21})\circ(E_{12}+E_{21})=I_2,
		\]
		we have
		\[
		e=\psi(I_2)
		=\psi(E_{12}+E_{21})\circ\psi(E_{12}+E_{21})
		=c\circ c.
		\]
		By the Peirce multiplication rule
		$\JJ_{12}\circ\JJ_{12}\subseteq\JJ_{11}\oplus\JJ_{22}$ from
		\eqref{eq:peirce-rules}, the element $e=c\circ c$ has zero
		$\JJ_{33}$-component. On the other hand, since $e_1,e_2,e_3$ are pairwise orthogonal idempotents, we have $e_i\in\JJ_{ii}$ for $i=1,2,3$.
		Therefore, in the simultaneous Peirce decomposition determined by
		$e_1,e_2,e_3$, the $\JJ_{33}$-component of $e=e_1+e_2+e_3$ is exactly $e_3$. Hence $e_3=0$ and consequently $e=e_1+e_2$.
	\end{proof}
	
	The next lemma is the point at which Theorem~\ref{thm:GKT} enters the proof. It transports one additivity relation between two linearly independent matrices to all pairs of matrices.
	
	\begin{lemma}\label{lem:transport}
		Let $\phi:M_n(\FF)^+\to\JJ$ be Jordan multiplicative, where $\JJ$ is a Jordan
		algebra over $\ZZ[1/2]$. Suppose that $U,V\in M_n(\FF)$ are linearly independent
		and
		\begin{equation}\label{eq:one-relation}
			\phi(U+V)=\phi(U)+\phi(V).
		\end{equation}
		Then $\phi$ is additive.
	\end{lemma}
	
	\begin{proof}
		Let $X,Y\in M_n(\FF)$. Choose an $\FF$-linear endomorphism $\Theta$ of
		$M_n(\FF)$ such that $\Theta(U)=X$ and $\Theta(V)=Y$. By
		Theorem~\ref{thm:GKT}, we may write $\Theta=L_{A_m}\cdots L_{A_1}$, for suitable matrices $A_1,\ldots,A_m\in M_n(\FF)$, with the empty product
		allowed. Define the corresponding additive operator on $\JJ$ by
		\[
		\Sigma:=L^\JJ_{\phi(A_m)}\cdots L^\JJ_{\phi(A_1)},
		\qquad
		L^\JJ_a(x)=a\circ x.
		\]
		A direct induction using \eqref{eq:JM} gives
		\[
		\phi(\Theta(Z))=\Sigma(\phi(Z)),
		\qquad Z\in M_n(\FF).
		\]
		Hence, the assumption and the additivity of $\Theta$ and $\Sigma$ give
		\begin{align*}
			\phi(X+Y)&=\phi(\Theta(U)+\Theta(V))=\phi(\Theta(U+V))
			=\Sigma(\phi(U+V))\stackrel{\eqref{eq:one-relation}}=\Sigma(\phi(U)+\phi(V))\\
			&=\Sigma(\phi(U))+\Sigma(\phi(V))          =\phi(\Theta(U))+\phi(\Theta(V))
			=\phi(X)+\phi(Y).
		\end{align*}
		Thus $\phi$ is additive.
	\end{proof}
	
	\begin{proposition}\label{prop:zero}
		Let $\phi:M_n(\FF)^+\to\JJ$, $n \ge 2$ be Jordan multiplicative, where $\JJ$ is a Jordan algebra over $\ZZ[1/2]$. If $\phi(0)=0$, then $\phi$ is additive and is either
		the zero map or injective.
	\end{proposition}
	
	\begin{proof}
		Restrict $\phi$ to the upper-left $2\times2$ corner $M_2(\FF)^+$ of
		$M_n(\FF)^+$. Proposition~\ref{prop:corner-relation} gives 
		\[
		\phi(E_{11}+E_{22})=\phi(E_{11})+\phi(E_{22}).
		\]
		Since $E_{11}$ and $E_{22}$ are linearly independent,
		Lemma~\ref{lem:transport} implies that $\phi$ is additive. The kernel of the additive Jordan homomorphism $\phi$ is a Jordan ideal of the
		Jordan ring $M_n(\FF)^+$. By Remark~\ref{rem:jordan-simple}, the kernel is either $0$ or all of
		$M_n(\FF)$. In the first case $\phi$ is injective, and in the second case it is
		the zero map.
	\end{proof}
	
	\begin{lemma}\label{lem:localization-zero}
		Let $\JJ$ be a $2$-torsion-free Jordan ring, and let
		$\phi:M_n(\FF)^+\to\JJ$, $n \ge 2$ be a Jordan multiplicative map such that $\phi(0) = 0$. Then
		$\phi$ is additive and is either zero or injective.
	\end{lemma}
	
	\begin{proof}
		Let $\iota:\JJ\to\JJ[1/2]$ be the injective localization map from
		Lemma~\ref{lem:localization-peirce}. The composite
		$\iota\circ \phi:M_n(\FF)^+\to\JJ[1/2]$ maps zero to zero and is Jordan
		multiplicative, so Proposition~\ref{prop:zero} applies. Thus $\iota\circ \phi$ is
		additive and is either zero or injective. Injectivity of $\iota$ gives the
		same result for $\phi$.
	\end{proof}
	
	\begin{proof}[Proof of Theorem~\ref{thm:general}]
		Suppose that $\phi$ is Jordan multiplicative. Since $0\circ0=0$, the element $e:=\phi(0)$ is idempotent. Also, for any $X \in M_n(\FF)$, $0\circ X=0$ gives $e\circ\phi(X)=e$. Thus, $e\circ(\phi(X)-e)=e-e\circ e=0$. Define
		\[
		\psi: M_n(\FF)^+ \to \JJ, \qquad \psi(X):=\phi(X)-e.
		\]
		The map $\psi$ takes its values in $\JJ_0(e)$. By
		Lemma~\ref{lem:localization-peirce}, $\JJ_0(e)$ is a Jordan subring of $\JJ$; since $\JJ$ is $2$-torsion-free, so is $\JJ_0(e)$. Furthermore $\psi(0)=0$. For all $X,Y\in M_n(\FF)$ we have
		\begin{align*}
			e+\psi(X\circ Y)
			&=\phi(X\circ Y)
			=\phi(X)\circ\phi(Y) =(e+\psi(X))\circ(e+\psi(Y))\\
			&=e+\psi(X)\circ\psi(Y).
		\end{align*}
		Hence
		$\psi(X\circ Y)=\psi(X)\circ\psi(Y)$, so $\psi$ is zero-preserving and Jordan
		multiplicative. By Lemma~\ref{lem:localization-zero}, $\psi$ is additive and is
		either zero or injective. Therefore $\phi$ is either constant equal to the
		idempotent $e$, or injective.
		
		The converse is immediate.
	\end{proof}
	
	\begin{corollary}\label{cor:commutative-codomain}
		Let $\mathcal C$ be a commutative associative $2$-torsion-free ring, regarded
		as a Jordan ring with its usual product. If $n\ge2$ and
		$\phi:M_n(\FF)^+\to\mathcal C$ is Jordan multiplicative, then $\phi$ is constant idempotent-valued. In particular, every Jordan multiplicative map $M_n(\FF)^+\to \mathcal D$ into an integral domain $\mathcal D$ of characteristic different from
		$2$ is either the constant zero map or the constant one map.
	\end{corollary}
	
	\begin{proof}
		By Theorem~\ref{thm:general}, we have $\phi=e+\psi$, where $e:=\phi(0)$ is an
		idempotent of $\mathcal C$ and $\psi:M_n(\FF)^+\to\mathcal C_0(e)$ is a Jordan ring homomorphism. Moreover, $\psi$ is either zero or injective.
		
		Assume that $\psi$ is injective. Since $\mathcal C$ is commutative and associative, so is every Jordan subring of $\mathcal C$. Hence the image of $\psi$ is associative as a Jordan ring, and injectivity of $\psi$ would imply that $M_n(\FF)^+$ is associative as a Jordan ring, which is a contradiction: for instance,
		\[
		(E_{12}\circ E_{21})\circ E_{11}=\tfrac12E_{11}
		\ne
		\tfrac14(E_{11}+E_{22})=E_{12}\circ(E_{21}\circ E_{11}),
		\]
		since $\chr(\FF)\ne2$. Thus $\psi=0$, and $\phi$ is constant idempotent-valued. If $\mathcal C=\mathcal D$ is an integral domain with $\chr(\mathcal C) \ne 2$, then the only idempotents are $0$ and $1$.
	\end{proof}

	\section{Associative and matrix codomains}\label{sec:assoc-matrix}
	
	We now specialize the main theorem to associative and matrix codomains. First
	observe the unital alternative. If $\JJ$ is a unital $2$-torsion-free Jordan
	ring and $\phi:M_n(\FF)^+\to\JJ$ is a unital Jordan multiplicative map, then
	Theorem~\ref{thm:general} implies that $\phi$ is either injective or is the
	constant map equal to the unit of $\JJ$. In the nonconstant case, the shifted
	map $\psi:=\phi-\phi(0)$ is a Jordan ring monomorphism. For associative
	codomains, this brings the classical theorem of Jacobson--Rickart into play:
	additive Jordan homomorphisms from full matrix rings decompose into
	homomorphic and antihomomorphic parts. We shall use the following consequence
	of their matrix-ring theorem \cite[Theorem~7]{JacobsonRickart}. It will be
	applied below to obtain the associative-codomain splitting and an explicit
	block diagonal form for full matrix codomains.
	
	\begin{corollary}[Jacobson--Rickart]\label{cor:JR}
		Let $\cA$ be a unital associative algebra over $\ZZ[1/2]$, and let $\psi:M_n(\FF)^+\to\cA^+$, $n\ge2$, be a unital Jordan ring homomorphism. Then there exists an idempotent $p$ in the associative subring generated by $\psi(M_n(\FF))$, central in that
		subring, such that
		\[
		X\mapsto p\psi(X),
		\qquad
		X\mapsto (1-p)\psi(X)
		\]
		are respectively a ring homomorphism and a ring antihomomorphism, mapping
		$I_n$ to $p$ and $1-p$.
	\end{corollary}
	
	\begin{proof}
		Since $\psi$ is additive and both $M_n(\FF)$ and $\cA$ are
		$\ZZ[1/2]$-modules, $\psi$ is $\ZZ[1/2]$-linear. Hence preservation of the special Jordan product \eqref{eq:J-product} gives, for all $X,Y\in M_n(\FF)$,
		\begin{align*}
			\psi(XY+YX)
			&=2\psi\!\left(\frac{XY+YX}{2}\right)  \\
			&=2\bigl(\psi(X)\circ\psi(Y)\bigr)
			=\psi(X)\psi(Y)+\psi(Y)\psi(X).
		\end{align*}
		By the equivalence stated by Jacobson--Rickart \cite[p.~479; see also p.~481]{JacobsonRickart}, it follows that $\psi$ is a Jordan homomorphism in their terminology. Therefore \cite[Theorem~7]{JacobsonRickart} applies, and $\psi$ is the sum of a homomorphism and an antihomomorphism. In the proof of that theorem, the summands are obtained by multiplication by complementary central
		idempotents in the associative subring generated by $\psi(M_n(\FF))$.
		Since $\psi$ is unital, this subring has identity $\psi(I_n)=1$. Writing
		these idempotents as $p$ and $1-p$ gives the stated form.
	\end{proof}
	
	\begin{proof}[Proof of Theorem~\ref{thm:associative}]
		By Theorem~\ref{thm:general}, $e:=\phi(0)$ is idempotent and
		$\psi:=\phi-e$ is a zero-preserving Jordan ring homomorphism into
		$(\cA^+)_0(e)$.  By Remark~\ref{rem:assoc-peirce-zero}, this Peirce 0-space
		is the associative corner $(1-e)\cA(1-e)$.
		
		Set $p:=\psi(I_n)$.  Then $p$ is an idempotent.  Since $I_n\circ X=X$, we have
		$p\circ\psi(X)=\psi(X)$ for all $X\in M_n(\FF)$ and hence $\psi(X)=p\psi(X)p$ by Remark~\ref{rem:assoc-peirce-zero}. Therefore,
		$\psi$ is a unital Jordan ring homomorphism $\psi:M_n(\FF)^+\to (p\cA p)^+$, where the corner $p\cA p$ is viewed as a unital associative algebra with identity $p$.
		
		Applying Corollary~\ref{cor:JR} to $\psi$ gives
		an idempotent $f$ in the associative subring of $p\cA p$ generated by
		$\psi(M_n(\FF))$, central in that subring, such that $\theta(X):=f\psi(X)$ and $\eta(X):=(p-f)\psi(X)$ are respectively a unital ring homomorphism $M_n(\FF)\to f\cA f$ and a unital ring antihomomorphism
		$M_n(\FF)\to(p-f)\cA(p-f)$.  Moreover, $\psi=\theta+\eta$.  Therefore
		$\phi=e+\theta+\eta$, as claimed.
	\end{proof}
	
	Let $\KK$ be another field of characteristic different from $2$. 
	We shall use the following standard normal form for finite-dimensional
	representations of full matrix rings. Let $\rho:M_n(\FF)\to M_m(\KK)$
	be a ring homomorphism. Then there are nonnegative integers $r,t$ with
	$nr+t=m$, an invertible matrix $S\in \GL_m(\KK)$, and, if $r>0$, a unital ring
	homomorphism $\omega:\FF\to M_r(\KK)$ such that
	\begin{equation}\label{eq:Morita-equivalence}
		S\rho(X)S^{-1}=\omega_n(X)\oplus 0_t,
		\qquad X\in M_n(\FF).
	\end{equation}
	Here $\omega_n$ is the \emph{$n$-fold amplification of $\omega$}: for
	$X=[x_{ij}]\in M_n(\FF)$,
	\[
	\omega_n(X):=[\omega(x_{ij})]\in M_n(M_r(\KK))\cong M_{nr}(\KK).
	\]
	In particular, $\omega_n:M_n(\FF)\to M_{nr}(\KK)$ is a ring homomorphism.
	If $r=0$, then $\rho=0$ and $t=m$; if $\rho$ is unital, then $t=0$.
	
	This is the familiar matrix-ring form of Morita equivalence
	\cite[Section~17B]{LamLectures}. In the same-field case $\KK=\FF$, this
	normal form is stated explicitly in
	\cite[Theorem~2.2]{LiTsaiWangWong}. For different codomain fields, the same
	matrix-unit proof applies verbatim: the diagonal matrix units of
	$M_n(\FF)$ decompose the unital part of the representation into $n$
	isomorphic $\KK$-vector spaces, and the action of the corner algebra $E_{11}M_n(\FF)E_{11}\cong \FF$ on one of these spaces gives the coefficient homomorphism $\omega:\FF\to M_r(\KK)$.
	
	\begin{proof}[Proof of Theorem~\ref{thm:matrix}]
		Apply Theorem~\ref{thm:associative} to $\phi$. Set
		$P:=\phi(0)$ and $\psi:=\phi-P$. Then $P$ is an idempotent,
		$\psi$ takes its values in the corner $(I_k-P)M_k(\KK)(I_k-P)$, and, with
		$Q:=\psi(I_n)$, there is an idempotent $R$ in the subring of
		$QM_k(\KK)Q$ generated by $\psi(M_n(\FF))$, central in that subring, such
		that
		\[
		\theta(X):=R\psi(X),
		\qquad
		\eta(X):=(Q-R)\psi(X), \qquad X \in M_n(\FF),
		\]
		are respectively a unital ring homomorphism into $RM_k(\KK)R$ and a unital
		ring antihomomorphism into $(Q-R)M_k(\KK)(Q-R)$, and
		\[
		\phi(X)=P+\theta(X)+\eta(X),
		\qquad X\in M_n(\FF).
		\]
		Since $\psi$ takes its values in $(I_k-P)M_k(\KK)(I_k-P)$, we have
		$PQ=QP=0$. Also $R\in QM_k(\KK)Q$, so $QR=RQ=R$. Hence $P$, $R$ and
		$Q-R$ are pairwise orthogonal idempotents. Thus, after a
		similarity over $\KK$, we may decompose
		\[
		\KK^k=\operatorname{im}P\oplus\operatorname{im}R\oplus\operatorname{im}(Q-R)\oplus\ker(P+Q).
		\]
		Relative to this decomposition,
		\[
		\phi(X)=I_s\oplus\theta(X)\oplus\eta(X)\oplus0_{t},
		\qquad X\in M_n(\FF),
		\]
		for some nonnegative integers $s$ and $t$, with summands of size $0$
		omitted. After identifying the corners
		$RM_k(\KK)R$ and $(Q-R)M_k(\KK)(Q-R)$ with full matrix algebras over $\KK$,
		we apply \eqref{eq:Morita-equivalence} to the homomorphism $\theta$. This
		homomorphism is unital with respect to the identity $R$ of its corner, since
		$\theta(I_n)=R$. We also apply \eqref{eq:Morita-equivalence} to the
		homomorphism
		\[
		M_n(\FF)\to (Q-R)M_k(\KK)(Q-R),
		\qquad X\mapsto \eta(X^{\mathsf T}),
		\]
		which is unital with respect to the identity $Q-R$ of its corner. Combining
		the resulting block similarities gives
		\[
		S\phi(X)S^{-1}
		=
		I_s\oplus\omega_n(X)\oplus\sigma_n(X^{\mathsf T})\oplus0_t,
		\qquad X\in M_n(\FF),
		\]
		for some nonnegative integers $r,u$ satisfying $s+nr+nu+t=k$, where
		$\omega:\FF\to M_r(\KK)$ and $\sigma:\FF\to M_u(\KK)$ are unital ring
		homomorphisms when the corresponding summands are present.
		
		The converse and the assertions about $\phi(0)$ and unitality follow immediately from this form.
	\end{proof}
	
	\begin{corollary}\label{cor:char-mismatch}
		If $\chr(\FF)\ne\chr(\KK)$, then every Jordan multiplicative map
		$M_n(\FF)^+\to M_k(\KK)^+$, $n\ge2$, is constant idempotent-valued.
	\end{corollary}
	
	\begin{proof}
		By Theorem~\ref{thm:matrix}, a nonconstant map would have a coefficient
		homomorphism $\tau:\FF\to M_r(\KK)$ for some $r\ge1$.  Since $\tau$ is unital, it is injective. For every positive integer $\ell$, we have $\tau(\ell 1_\FF)=\ell I_r$, and consequently
		\[
		\ell 1_\FF=0
		\quad\Longleftrightarrow\quad
		\ell I_r=0,
		\]
		where the reverse implication uses injectivity of $\tau$.  Thus $\chr(\FF)=\chr(M_r(\KK))=\chr(\KK)$, contrary to the hypothesis. Hence no coefficient homomorphism can occur in the normal form
		\eqref{eq:normal}.  Therefore $r=u=0$, and
		$S\phi(X)S^{-1}=I_s\oplus0_t$ for every $X\in M_n(\FF)$.  Thus $\phi$ is
		constant idempotent-valued.
	\end{proof}
	
	For fields $\FF$ and $\KK$, let $d_\KK(\FF)$ be the least integer $r\ge1$ for
	which there exists a unital ring homomorphism $\FF\to M_r(\KK)$; if no such
	$r$ exists, set $d_\KK(\FF)=\infty$.
	
	\begin{corollary}\label{cor:threshold}
		Let $\FF$ and $\KK$ be fields of characteristics different from $2$, and let
		$n\ge2$ and $k\ge1$. There exists a nonconstant Jordan multiplicative map
		$M_n(\FF)^+\to M_k(\KK)^+$ if and only if $k\ge n d_\KK(\FF)$, where the
		inequality is understood to be false when $d_\KK(\FF)=\infty$. Equivalently,
		if $k<n d_\KK(\FF)$, then every Jordan multiplicative map
		$M_n(\FF)^+\to M_k(\KK)^+$ is constant idempotent-valued.
	\end{corollary}
	
	\begin{proof}
		If a nonconstant map exists, then at least one homomorphic or
		antihomomorphic block occurs in \eqref{eq:normal}. Hence a coefficient
		homomorphism occurs, its coefficient size is at least $d_\KK(\FF)$ by
		definition, and therefore $k\ge n d_\KK(\FF)$. Conversely, if
		$d:=d_\KK(\FF)<\infty$ and $k\ge nd$, choose a unital homomorphism
		$\theta:\FF\to M_d(\KK)$. Then
		$X\mapsto\theta_n(X)\oplus0_{k-nd}$ is a nonconstant Jordan multiplicative
		map. This proves the equivalence, and the final assertion follows.
	\end{proof}
	
	\begin{remark}\label{rem:coefficients}
		The coefficient homomorphisms in Theorem~\ref{thm:matrix} have the usual
		module-theoretic interpretation. If $\chr(\FF)\ne\chr(\KK)$, no such
		homomorphism exists. If $\FF$ and $\KK$ have common prime subfield $\PP$, then, for each $r\ge1$, similarity classes of unital homomorphisms $\FF\to M_r(\KK)$ are in natural
		bijection with isomorphism classes of unital $\KK\otimes_{\PP}\FF$-modules of
		$\KK$-dimension $r$, where $\KK$ acts through the first tensor factor. This is
		the standard passage between module structures and change of basis; see,
		e.g., \cite[pp.~199, 201--202]{Morrison}. Namely, a unital homomorphism
		$\rho:\FF\to M_r(\KK)$ defines the action
		\[
		(a\otimes x)v:=a\,\rho(x)v,
		\qquad a\in\KK,\ x\in\FF,\ v\in\KK^r.
		\]
		Conversely, the action of $1\otimes x$ gives a unital homomorphism
		$\FF\to\End_\KK(V)$; choosing a $\KK$-basis identifies this with a
		homomorphism $\FF\to M_r(\KK)$, and a change of basis conjugates it.
		
		Suppose that $\FF/\PP$ is finite separable. Choose a primitive element $\alpha\in\FF$ \cite[Chapter~V, Theorem~4.6]{LangAlg}, let $m_\alpha(x)\in\PP[x]$ be its minimal polynomial, and write $m_\alpha(x)=g_1(x)\cdots g_s(x)$ over $\KK$, where the $g_i$ are distinct
		monic irreducible polynomials. Then the standard scalar-extension and Chinese-remainder decompositions give
		\[
		\KK\otimes_{\PP}\FF
		\cong
		\KK[x]/(m_\alpha(x))
		\cong
		\prod_{i=1}^s \KK[x]/(g_i(x));
		\]
		see \cite[Chapter~XVI, Exercise~2, and Chapter~II, Corollary~2.2]{LangAlg}. Therefore $d_\KK(\FF)=\min_{1\le i\le s}\deg g_i$.
		
		For finite fields, let $\FF=\mathbb F_{p^a}$ and $\KK=\mathbb F_{p^b}$, 
		with $p$ odd. Choose $\alpha\in\mathbb F_{p^a}$ such that 
		$\mathbb F_{p^a}=\mathbb F_p(\alpha)$, and let $m_\alpha\in\mathbb F_p[x]$ 
		be its minimal polynomial, of degree $a$. By \cite[Theorem~3.46]{LidlNiederreiter}, the polynomial $m_\alpha$ factors over $\mathbb F_{p^b}$ into $\gcd(a,b)$ distinct irreducible factors, each of degree $a/\gcd(a,b)$. Therefore $d_{\mathbb F_{p^b}}(\mathbb F_{p^a})=a/\gcd(a,b)$. By
		Corollary~\ref{cor:threshold}, nonconstant Jordan multiplicative maps
		$M_n(\mathbb F_{p^a})^+\to M_k(\mathbb F_{p^b})^+$ exist if and only if
		$k\ge n a/\gcd(a,b)$.
	\end{remark}
	
	\begin{remark}\label{rem:exceptional-and-other-domains}
		It is natural to ask whether analogous rigidity statements hold for other
		finite-dimensional simple Jordan algebras. The proof given here uses two
		features specific to $M_n(\FF)^+$: the Peirce calculation in the upper-left
		copy of $M_2(\FF)^+$ and Theorem~\ref{thm:GKT}, which identifies the
		semigroup generated by the Jordan multiplication operators on $M_n(\FF)^+$ with the full endomorphism semigroup $\End_\FF(M_n(\FF))$. We plan to address these questions in subsequent
		work.
	\end{remark}

	\section*{Funding and AI-use Disclosure}
	This research was supported by the European Union -- NextGenerationEU through the National Recovery and Resilience Plan 2021--2026 Institutional grants of University of Zagreb Faculty of Science (IK IA 1.1.3. Impact4Math, PMF-CROFUND).
	
	M.K. was supported by the Croatian Science Foundation under the project no.\ IP-2022-10-5008 (TEBAG) and acknowledges support from the project “Implementation of cutting-edge research and its application as part of the Scientific Center of Excellence for Quantum and Complex Systems, and Representations of Lie Algebras”, Grant No.\ PK.1.1.10.0004, co-financed by the European Union through the European Regional Development Fund -- Competitiveness and Cohesion Programme 2021-2027. 
	
	The authors disclose that OpenAI's GPT-5.5 was used for copyediting, language polishing, and preliminary bibliographic searches. The authors reviewed all AI-assisted suggestions and take full responsibility for the final manuscript.

\end{document}